\documentclass[12pt,leqno]{article}
\usepackage{amssymb}
\usepackage[mathscr]{eucal}
\usepackage{amsmath,amssymb,latexsym,theorem,enumerate,ifthen}
\usepackage[english]{babel}
\usepackage{color,url}
\usepackage{appendix}
\usepackage{graphicx}
\usepackage{subcaption}

\newcommand{\NN}{\mathbb{N}}
\newcommand{\QQ}{\mathbb{Q}}
\newcommand{\RR}{\mathbb{R}}

\newcommand{\ZZ}{\mathbb{Z}}

\newcommand{\bone}{{\boldsymbol{1}}}

\newcommand{\cA}{{\mathcal A}}

\newcommand{\ee}{\mathrm{e}}

\newcommand{\EE}{\operatorname{\mathbb{E}}}

\newcommand{\PP}{\operatorname{\mathbb{P}}}

\newcommand{\comment}[1]{}

\renewcommand{\leq}{\leqslant}
\renewcommand{\geq}{\geqslant}

\newcommand{\proofend}{\hfill\mbox{$\Box$}}

\def\comment#1{}

\numberwithin{equation}{section}

\theoremstyle{change} \theorembodyfont{\em}
\newtheorem{Lem}{Lemma.}[section]
\newtheorem{Thm}[Lem]{Theorem.}
\newtheorem{Pro}[Lem]{Proposition.}

\theorembodyfont{\rm}
\newtheorem{Def}[Lem]{Definition.}
\newtheorem{Rem}[Lem]{Remark.}

\def\OnlyOnArXiv#1#2{\ifthenelse{\equal{#1}{Y}}{#2}{}}

\long\def\Eq#1#2{\ifthenelse{\equal{#1}{*}}
  {\begin{equation*}\begin{aligned}#2\end{aligned}\end{equation*}}
  {\begin{equation}\begin{aligned}\label{#1}#2\end{aligned}\end{equation}}}

\newenvironment{proof}{\noindent{\bf Proof.}}{\proofend}

\begin{document}

\begin{center}
 {\bfseries\Large A convexity-type functional inequality\\[1mm] with infinite convex combinations}

\vspace*{3mm}

{\sc\large
  M\'aty\'as $\text{Barczy}^{*}$,
  Zsolt $\text{P\'ales}^{**,\diamond}$\\[3mm] }

\vskip0.3cm

\centerline{\it Dedicated to the memory of Professor Ferenc Schipp}
\end{center}

\vskip0.2cm

\noindent
 * HUN-REN–SZTE Analysis and Applications Research Group,
   Bolyai Institute, University of Szeged,
   Aradi v\'ertan\'uk tere 1, H--6720 Szeged, Hungary.

\noindent
 ** Institute of Mathematics, University of Debrecen,
    Pf.~400, H--4002 Debrecen, Hungary.

\noindent E-mails: barczy@math.u-szeged.hu (M. Barczy),
                  pales@science.unideb.hu  (Zs. P\'ales).

\noindent $\diamond$ Corresponding author.

\vskip0.2cm


{
\renewcommand{\thefootnote}{}
\vspace*{0.2cm}
\footnote{Zsolt P\'ales is supported by the K-134191 NKFIH Grant.}}

\vspace*{-10mm}

\begin{abstract}
Given a function $f$ defined on a nonempty and convex subset of the $d$-dimensional Euclidean space, we prove that if $f$ is bounded from below and it satisfies a convexity-type functional inequality with infinite convex combinations, then $f$ has to be convex.
We also give alternative proofs of a generalization of some known results on convexity with infinite convex combinations due to Dar\'oczy and P\'ales (1987) and Pavi\'c (2019) using a probabilistic version of Jensen inequality.
\end{abstract}


\section{Introduction}
\label{section_intro}

Studying properties of convex functions, in particular inequalities for them, has a long history, and it is still an active area of research, 
 mainly because of their importance in many fields of mathematics.
For a recent monograph on some developments in the theory of Jensen-type inequalities 
 (such as Jensen--Steffensen, Jensen--Mercer, Jessen, McShane and Popoviciu inequalities), 
 and for applications in information theory, see Khan et al.\ \cite{KhaKahPecPec}. 
When convexity of a function comes into play, one usually means convexity with finite convex combinations.
In this paper, we deal with a convexity-type functional inequality with infinite convex combinations, which, in the language of probability theory, can be rephrased as Jensen inequality for discrete random variables.  
This sub-field of the theory of convex functions seems to be less investigated than that of convexity for finite convex combinations.

Throughout this paper, let $\NN$, $\ZZ_+$, $\RR$ and $\RR_+$ denote the sets of positive integers, non-negative integers, real numbers and  non-negative real numbers, respectively.
For a real number $z\in\RR$, its positive part $\max(z,0)$ is denoted by $z_+$.
For a vector $x\in\RR^d$, its Euclidean norm is denoted by $\Vert x\Vert$.
An interval $I\subseteq\RR$ will be called nondegenerate if it contains at least two distinct points.

\begin{Def}\label{Def_convexity}  Let $d\in\NN$ and $D\subseteq \RR^d$ be a nonempty convex set.
\vspace*{-3mm}
\begin{enumerate}
 \item[(i)] Given $t\in[0,1]$, a function $f:D\to\RR$ is called $t$-convex if
        \[
          f(tx+(1-t)y) \leq tf(x) + (1-t)f(y), \qquad x,y\in D.
        \]
        If $f$ is $\frac{1}{2}$-convex, then it is called Jensen convex or midpoint convex as well. If $f$ is $t$-convex for all $t\in(0,1)$, then it is said to be convex.
        Note that every function $f:D\to\RR$ is automatically $0$-convex and $1$-convex as well.
 \item[(ii)] Given $t,s\in[0,1]$, a function $f:D\to\RR$ is called $(t,s)$-convex if
                 \[
                 f(tx+(1-t)y) \leq sf(x) + (1-s)f(y), \qquad x,y\in D.
                 \]  
 \end{enumerate}
\end{Def}

The following result establishes implications among the convexity notions introduced in part (i) of Definition \ref{Def_convexity}.

\begin{Thm}\label{Thm_Convex}
Let $d\in\NN$, $D\subseteq \RR^d$ be a nonempty convex set and $f:D\to\RR$.
Then the following two assertions hold.
\begin{enumerate}[(i)]
 \item If $f$ is $t$-convex for some $t\in(0,1)$, then it is Jensen convex.
 \item If $D$ is open as well, $f$ is $t$-convex for some $t\in(0,1)$ and $f$ is bounded from above on some nonempty open subset of $D$, then $f$ is convex and continuous.
\end{enumerate}
\end{Thm}

\begin{proof}
For the proof of part (i), see Daróczy and Páles \cite[Lemma 1]{DarPal87} or Kuhn \cite{Kuh84}.
(For completeness, we note that the compactness assumption of the domain of $f$ in Lemma 1 in Daróczy and Páles \cite{DarPal87}
 is not needed for the validity of their result, so we can indeed apply it.)

Next, we prove part (ii). Using part (i), we have that $f$ is a Jensen convex function defined on the nonempty open and convex set $D$, 
 and, by the assumption, it is bounded from above on some nonempty open subset of $D$.
By Theorem 6.2.1 in Kuczma \cite{Kuc2009}, it implies that $f$ is locally bounded at every point of $D$ as well.
As a consequence, the celebrated Bernstein--Doetsch theorem (see Bernstein and Doetsch \cite{BerDoe15} or Kuczma \cite[Theorems 6.4.2 and 7.1.1]{Kuc2009}) yields that $f$ is convex and continuous on $D$, as desired.
\end{proof}

Now we recall a result due to Dar\'oczy and P\'ales \cite{DarPal87} about infinite convex combinations.
Let us introduce the set
 \[
   \Lambda:= \Bigg\{ (\lambda_n)_{n\in\NN} : \lambda_n\geq \lambda_{n+1}>0, n\in\NN, \;\text{and}\; \sum_{n=1}^\infty \lambda_n=1\Bigg\}.
 \]

\begin{Thm}[Dar\'oczy and P\'ales {\cite[Theorem 1]{DarPal87}}]\label{ThmDarPal}
Let $(\lambda_n)_{n\in\NN}\in\Lambda$, $d\in\NN$, $D\subseteq\RR^d$ be a compact, convex set, and $f:D\to\RR_+$.
Then the inequality
 \begin{align*}
  f\left( \sum_{i=1}^\infty \lambda_i x_i \right)
     \leq \sum_{i=1}^\infty \lambda_i f(x_i)
 \end{align*}
 holds for all $x_i\in D$, $i\in\NN$, if and only if $f$ is convex.
\end{Thm}

Pavi\'c \cite{Pav2019} recently has proven the following similar result.

\begin{Thm}[Pavi\'c {\cite[Theorem 2.1]{Pav2019}}]\label{Pav2019}
Let $f:[a,b]\to\RR$ be a convex function, where $a,b\in\RR$ with $a<b$.
Let $x_i\in[a,b]$, $i\in\NN$, and $\lambda_i\in\RR_+$, $i\in\NN$, be such that $\sum_{i=1}^\infty \lambda_i=1$.
Then
 \[
   f(t a+ (1-t)b)=f\left( \sum_{i=1}^\infty \lambda_i x_i\right) 
    \leq \sum_{i=1}^\infty \lambda_i f(x_i) \leq t f(a) + (1-t) f(b),
 \]
 where $t\in[0,1]$ is such that $\sum_{i=1}^\infty \lambda_ix_i = t a+(1-t) b$.
\end{Thm}

As the main result of the paper, given a function $f$ defined on a nonempty and convex subset of $\RR^d$, 
  we prove that if $f$ is bounded from below and it satisfies a convexity-type functional inequality 
  with infinite convex combinations, then $f$ has to be convex, see Theorem \ref{Thm_convex_infinity}.
We also give alternative proofs for generalizations of some parts of Theorems \ref{ThmDarPal} 
 and \ref{Pav2019} due to Dar\'oczy and P\'ales \cite{DarPal87} and Pavi\'c \cite{Pav2019}, respectively,
 see Proposition \ref{Pro_altalanositas}.
In the proof of Theorem \ref{Thm_convex_infinity}, results on $(s,t)$-convexity, while in the proof of 
  Proposition \ref{Pro_altalanositas}, a probabilistic version Jensen inequality (see Section \ref{Sec_Jensen}) play a crucial role.
In Remark \ref{Rem_Thm_no_converse}, we give an example which shows that a converse of Theorem \ref{Thm_convex_infinity} 
 does not hold in general.

\section{A probabilistic version of the Jensen inequality}\label{Sec_Jensen}

In this section, we recall a probabilistic version of Jensen inequality (see, e.g., Dudley \cite[10.2.6, page 348]{Dud}), which will be used for giving alternative proofs for generalizations of some parts of Theorems \ref{ThmDarPal} and \ref{Pav2019}.
  
\begin{Pro}[Jensen inequality]\label{Pro_Jensen_multi}
Let $\xi : \Omega \to \RR^d$ be a random vector such that $\EE(\|\xi\|) < \infty$, and let $D\subseteq \RR^d$ be a Borel measurable, convex set such that $\PP(\xi \in D)=1$. 
Then the following statements hold.
  \begin{itemize}
    \item[(i)] We have $\EE(\xi) \in D$.
    \item[(ii)] For all continuous and convex functions $f : D \to \RR$, we have that 
           $\EE(f(\xi))\in(-\infty,\infty]$ and $f(\EE(\xi)) \leq \EE(f(\xi))$.  
    \item[(iii)] Supposing, in addition, that $\xi$ is discrete as well, for all convex functions $f : D \to \RR$, we have that $\EE(f(\xi))\in(-\infty,\infty]$ and $f(\EE(\xi)) \leq \EE(f(\xi))$. 
  \end{itemize}
\end{Pro}

\begin{proof}
We plan to apply Dudley \cite[10.2.6, page 348]{Dud} to derive the assertions.
First note that, for all continuous functions $f : D \to \RR$, we have that $f(\xi)$ is a random variable.
Further, if $\xi$ is discrete as well, then, for all functions $f : D \to \RR$, we have that 
 $f(\xi)$ is a random variable.
Indeed, if $\{x_i:i\in\NN\}$ denotes the range of $\xi$, then the range of $f(\xi)$ is
 $\{f(x_i): i\in\NN\}$ is also countable, and, for each $j\in\NN$, we get that 
 \[
  \{ f(\xi)=f(x_j)\} = \bigcup_{\{ i\in\NN : f(x_i) = f(x_j)\}} \{\xi=x_i\},
 \]
  which is an event, since $\{\xi=x_i\}$ is an event for each $i\in\NN$ and a $\sigma$-algebra
  is closed under countable union.
Consequently, if $\xi(\omega)\in D$ for all $\omega\in\Omega$, then (i)-(iii) directly follow from
 Dudley \cite[10.2.6, page 348]{Dud}.

In the more general case when $\PP(\xi \in D)=1$ holds, let us introduce the mapping $\widetilde\xi:\Omega\to\RR^d$,
 \[
  \widetilde\xi(\omega):= \big(\xi\bone_{\{\xi\in D\}} + d_0 \bone_{\{\xi\not\in D\}}\big)(\omega)
                         =\begin{cases}
                           \xi(\omega) & \text{if $\omega\in\Omega$ is such that $\xi(\omega)\in D$,}\\
                           d_0 & \text{otherwise,}
                          \end{cases}                          
 \]
 where $d_0$ is arbitrarily fixed element of $D$.
Then $\widetilde\xi(\omega)\in D$ for all $\omega\in\Omega$, and, since $\PP(\xi\not\in D)=0$, we have 
  \begin{align*}
    \EE(\|\widetilde \xi\|) &= \EE\big(\|\widetilde \xi\| \bone_{\{\xi\in D\}}\big) + \EE\big(\|\widetilde \xi\| \bone_{\{\xi\not\in D\}}\big) \\
                            &= \EE\big(\|\xi\| \bone_{\{\xi\in D\}}\big) + \EE\big(\|d_0\| \bone_{\{\xi\not\in D\}}\big) \\
                            &=\EE(\|\xi\|) - \EE\big(\|\xi\| \bone_{\{\xi\not\in D\}}\big)  + \|d_0\|\PP(\xi\not\in D) = \EE(\|\xi\|)<\infty.
  \end{align*}
Using again Dudley \cite[10.2.6, page 348]{Dud}, we obtain that $\EE(\widetilde\xi)\in D$,
  $\EE\big(f(\widetilde\xi)\big)\in(-\infty,\infty]$ and $f\big(\EE(\widetilde \xi)\big) \leq \EE\big(f(\widetilde \xi)\big)$.
Since $\PP(\xi\not\in D)=0$, we get that
 \begin{align*}
   \EE(\widetilde\xi)& = \EE\big(\widetilde\xi \bone_{\{\xi\in D\}}\big)+\EE\big(\widetilde\xi \bone_{\{\xi\not\in D\}}\big)
                       =  \EE\big(\xi \bone_{\{\xi\in D\}}\big)+\EE\big(d_0 \bone_{\{\xi\not\in D\}}\big) \\
                     & = \EE(\xi) - \EE\big(\xi \bone_{\{\xi\not\in D\}}\big) + d_0\PP(\xi\not\in D)
                      = \EE(\xi), 
 \end{align*}
 and therefore $\EE(\xi)\in D$ holds as well. 
Similarly, one can check that $\EE\big(f(\widetilde\xi)\big) = \EE\big(f(\xi)\big)$, 
 and consequently (i)-(iii) also follow in the case of $\PP(\xi \in D)=1$.
\end{proof}

Note that if $d>1$ and $D\subseteq \RR^d$ is a convex set, then $D$ may be not Borel measurable. Therefore, the assumption about the Borel measurability of $D$ in the above proposition is indispensable.
Namely, the union of an open ball in $\RR^d$ with any non Borel measurable subset of its boundary provides an example of a non Borel measurable convex set of $\RR^d$, cf.\ Lang \cite{Lang}.

\section{Main results}

First, we give alternative proofs for a generalized form of the 'only if' part of Theorem \ref{ThmDarPal} and the first inequality of Theorem \ref{Pav2019} using a probabilistic version of Jensen inequality (see part (iii) of Proposition \ref{Pro_Jensen_multi}).

\begin{Pro}\label{Pro_altalanositas}
Let $\lambda_i\in\RR_+$, $i\in\NN$, be such that $\sum_{i=1}^\infty \lambda_i =1$.
Let $d\in\NN$, $D\subseteq\RR^d$ be a convex set, and let $x_i\in D$, $i\in\NN$, be such that 
$\sum_{i=1}^\infty \lambda_i \Vert x_i \Vert <\infty$.
Then, for any convex function $f:D\to\RR$, we have that 
\[
  f\left( \sum_{i=1}^\infty \lambda_i x_i \right)
     \leq \sum_{i=1}^\infty  \lambda_i f(x_i).  
\]
\end{Pro}

\begin{proof}
Let $\xi$ be a discrete random variable with range $\{x_i: i\in\NN\}$ and distribution $\PP(\xi=x_i)=\lambda_i$, $i\in\NN$.
Since $\lambda_i\in\RR_+$, $i\in\NN$, and $\sum_{i=1}^\infty \lambda_i =1$, the distribution of $\xi$ is indeed well-defined.
Then, by the assumption, we have that
 \[
   \EE(\Vert \xi\Vert) = \sum_{i=1}^\infty \lambda_i \Vert x_i \Vert<\infty.
 \] 
Consequently, by part (i) of Proposition \ref{Pro_Jensen_multi}, we get $\EE(\xi)\in D$, i.e., 
\[
 \sum_{i=1}^\infty \lambda_ix_i=\sum_{i=1}^\infty x_i\PP(\xi=x_i)=\EE(\xi)\in D.
\]
Furthermore, by part (iii) of Proposition \ref{Pro_Jensen_multi}, for any convex function $f:D\to\RR$, we have that $f(\EE(\xi))\leq\EE(f(\xi))$, i.e., 
  \[
     f\left( \sum_{i=1}^\infty \lambda_i x_i \right)=f(\EE(\xi))
     \leq\EE(f(\xi))=\sum_{i=1}^\infty f(x_i)\PP(\xi=x_i)
     =\sum_{i=1}^\infty  \lambda_i f(x_i).
  \]
This completes the proof of the statement.
\end{proof} 

Next, we present a generalization of the 'if part' of Theorem \ref{ThmDarPal}.

\begin{Thm}\label{Thm_convex_infinity}
Let $(\lambda_n)_{n\in\NN},(\mu_n)_{n\in\NN}\in\Lambda$,
let $d\in\NN$, $D\subseteq\RR^d$ be nonempty and convex,
 and let $f:D\to\RR$ be bounded from below.
If the inequality
 \begin{align}\label{help_convex_infinity}
  f\left( \sum_{i=1}^\infty \lambda_i x_i \right)
     \leq \sum_{i=1}^\infty \mu_i f(x_i)
 \end{align}
 holds for all bounded sequences $x_i\in D$, $i\in\NN$, then $f$ is convex.
Here the right hand side of \eqref{help_convex_infinity} is either convergent, or else, it diverges
to the extended real number $+\infty$.
\end{Thm}

The inequality \eqref{help_convex_infinity} can be reformulated as 
\[
  f(\EE(\xi)) \leq \EE(f(\eta)),
\]
 where $(\Omega,\cA,\PP)$ is a probability space,
 $\xi:\Omega\to\RR^d$ and $\eta:\Omega\to\RR^d$
 are random vectors with common ranges $\{x_i: i\in\NN\}$ and with possibly different distributions
 $\PP(\xi=x_i)=\lambda_i$, $i\in\NN$, and $\PP(\eta=x_i)=\mu_i$, $i\in\NN$, respectively.
 
\smallskip

\noindent{\bf Proof of Theorem \ref{Thm_convex_infinity}.}
First, we check that the left hand side of \eqref{help_convex_infinity} is well-defined for all bounded sequences $x_i\in D$, $i\in\NN$, that is, 
 $\sum_{i=1}^\infty \lambda_i x_i$ belongs to $D$.
Indeed, one can interpret the sum $\sum_{i=1}^\infty \lambda_i x_i$ as the expectation of a discrete random vector $\xi$ with range $\{x_i: i\in\NN\}$ and with distribution $\PP(\xi=x_i)=\lambda_i$, $i\in\NN$.
Since $(x_i)_{i\in\NN}$ is bounded, there exists $K>0$ such that $\Vert x_i\Vert\leq K$, $i\in\NN$, yielding that
 \[
 \sum_{i=1}^\infty \Vert \lambda_i x_i\Vert = \sum_{i=1}^\infty \lambda_i \Vert x_i\Vert
  \leq K \sum_{i=1}^\infty\lambda_i = K<\infty.
 \]
As a consequence, we have that $\xi$ is integrable, and, by part (i) of Proposition \ref{Pro_Jensen_multi},
we get that $\sum_{i=1}^\infty \lambda_i x_i=\EE(\xi)\in D$.

Next, we show that the right hand side of \eqref{help_convex_infinity} is either convergent, or else, it diverges to the extended real number $+\infty$.
Assume that $B$ is a lower bound for $f$. 
Then $f-B\geq0$, and hence the series $\sum_{i=1}^\infty \mu_i (f(x_i)-B)$ is either convergent, 
  or else, it diverges to the extended real number $+\infty$. Therefore, observing that
\[
  \sum_{i=1}^\infty \mu_i f(x_i)
  =\sum_{i=1}^\infty \mu_i (f(x_i)-B)+\sum_{i=1}^\infty (\mu_i B)=\sum_{i=1}^\infty \mu_i (f(x_i)-B)+B,
\]
we can conclude that the series $\sum_{i=1}^\infty \mu_i f(x_i)$ is also either convergent, or else, it diverges to the extended real number $+\infty$.

Now, suppose that \eqref{help_convex_infinity} holds for all bounded sequences $x_i\in D$, $i\in\NN$.
Let $x,y\in D$ be fixed.
By choosing $x_1:=x$ and $x_i:=y$, $i\in\NN\setminus\{1\}$, in \eqref{help_convex_infinity}, we have that
 \begin{align*}
  f(\lambda_1x+(1-\lambda_1)y)
    &=  f\left(\lambda_1x+ \left(\sum_{i=2}^\infty\lambda_i \right)y\right)\\
    &\leq  \mu_1 f(x)
       + \left( \sum_{i=2}^\infty \mu_i \right) f(y)
     = \mu_1 f(x) + ( 1- \mu_1) f(y).
 \end{align*}
Since $\lambda_i>0$ and $\mu_i>0$ for each $i\in\NN$ and $\sum_{i=1}^\infty \lambda_i = \sum_{i=1}^\infty \mu_i=1$,
 we have that $\lambda_1,\mu_1\in(0,1)$.
Consequently, we get that $f$ is $(\lambda_1,\mu_1)$-convex on $D$.
By Fact on page 135 in Matkowski and Pycia \cite{MatPyc}, we get that $f$ is Jensen convex on $D$.
Here, for historical fidelity, we mention that this result is due to Kuhn \cite{Kuh87} and, then,
 using the idea in the proof of Lemma 1 in Dar\'oczy and P\'ales \cite{DarPal87}, later
 Matkowski and Pycia \cite{MatPyc} gave an elementary proof of the result of Kuhn \cite{Kuh87}.
This together with Theorem 5.3.5 in Kuczma \cite{Kuc2009}
 imply that $f$ is $q$-convex for all $q\in\QQ\cap[0,1]$.
(For completeness, we note that the openness of $D$ in Theorem 5.3.5 in Kuczma \cite{Kuc2009}
 is assumed, but, in fact, it is not needed for the validity of the result.)

By Lemma 2 in Dar\'oczy and P\'ales \cite{DarPal87},
 for any $t\in[0,1]$, there exist $q_i\in \QQ\cap[0,1]$, $i\in\NN$, such that
 $t=\sum_{i=1}^\infty \lambda_iq_i$, and hence
 \[
   1-t = \sum_{i=1}^\infty \lambda_i - \sum_{i=1}^\infty \lambda_iq_i
       = \sum_{i=1}^\infty \lambda_i(1-q_i).
 \]
Therefore, for any $x,y\in D$ and $t\in[0,1]$, we get that
 \begin{align*}
   f(tx+(1-t)y)
      = f\left( \left(\sum_{i=1}^\infty \lambda_iq_i \right)x+ \left(\sum_{i=1}^\infty \lambda_i(1-q_i)\right)y\right)
      = f\left(\sum_{i=1}^\infty \lambda_i (q_i x + (1-q_i)y)\right).
 \end{align*}
Using \eqref{help_convex_infinity} for the bounded sequence $q_i x + (1-q_i)y\in D$, $i\in \NN$, and that $f$ is 
 $q_i$-convex for each $i\in\NN$, we have that
 \begin{align*}
   f(tx+(1-t)y) & \leq \sum_{i=1}^\infty \mu_i f(q_i x + (1-q_i)y)
                 \leq \sum_{i=1}^\infty \mu_i ( q_if(x) + (1-q_i)f(y) )\\
                & = \left(  \sum_{i=1}^\infty \mu_i q_i \right) f(x)
                   + \left(  \sum_{i=1}^\infty \mu_i (1-q_i) \right) f(y)\\
                & \leq \left(  \sum_{i=1}^\infty \mu_i q_i \right) (f(x))_+
                   + \left(  \sum_{i=1}^\infty \mu_i (1-q_i) \right) (f(y))_+\\
                & \leq \left(  \sum_{i=1}^\infty \mu_i \right) (f(x))_+
                   + \left(  \sum_{i=1}^\infty \mu_i \right) (f(y))_+\\
                & = (f(x))_++(f(y))_+ .
 \end{align*}
This proves that $f$ is bounded from above on the segment $[x,y]:=\{rx+(1-r)y: r\in[0,1]\}$ connecting $x$ and $y$.
All in all, $f:D\to\RR$ is a Jensen convex function defined on a non-empty and convex set $D\subseteq \RR^d$
 such that, for all $x,y\in D$, it is bounded from above on the segment $[x,y]$.
To finish the proof, for any $x,y\in D$ with $x\ne y$, let us apply the Bernstein--Doetsch theorem 
 (see Bernstein and Doetsch \cite{BerDoe15} or Kuczma \cite[Theorem 6.4.2]{Kuc2009})
 to the convex and open set $(x,y):=\{rx+(1-r)y: r\in(0,1)\}$ and the Jensen convex function
 $f_{x,y}:(x,y)\to\RR$, $f_{x,y}(v):=f(v)$, $v\in(x,y)$, which is bounded from above.
Then, for all $x,y\in D$ with $x\ne y$, we have that $f_{x,y}$ is continuous and convex on $(x,y)$.
Since $f_{x,y}$ is a restriction of $f$ onto $(x,y)$, it readily yields that $f$ is convex on $D$, as desired.
\proofend

In the next remark, we provide an example, which shows that a converse of Theorem \ref{Thm_convex_infinity} does not hold in general.

\begin{Rem}\label{Rem_Thm_no_converse}
Let $(\lambda_n)_{n\in\NN}\in\Lambda$ and $(\mu_n)_{n\in\NN}\in\Lambda$ be such that 
 $(\ee^{\lambda_1}-1)/(\ee-1)>\mu_1$.
Let $d:=1$, $D:=\RR$ and $f:\RR\to\RR$, $f(x):=\ee^{x}$, $x\in\RR$.
Then $f$ is convex, and, with $x_1:=1$ and $x_i:=0$, $i\in\NN\setminus\{1\}$, we get that 
 \[
  f\left(\sum_{i=1}^\infty \lambda_i x_i \right) = f(\lambda_1) = \ee^{\lambda_1},
 \]  
 and 
 \[
  \sum_{i=1}^\infty \mu_i f(x_i) = \mu_1 f(1) + (1-\mu_1)f(0) 
                                 = \mu_1\ee + 1-\mu_1
                                 = 1+\mu_1(\ee-1).
 \]  
Since $(\ee^{\lambda_1}-1)/(\ee-1)>\mu_1$, we have that the inequality \eqref{help_convex_infinity}
 does not hold. 
\proofend
\end{Rem}

\bibliographystyle{plain}

\end{document}